\documentclass[10pt, a4paper, twocolumn ,titlepage, egregdoesnotlikesansseriftitles]{scrartcl}
\pdfoutput=1
\usepackage{geometry}
\usepackage{mathtools} 
\usepackage{stmaryrd} 
\usepackage{amsmath} 
\usepackage{amsthm} 
\usepackage{amssymb} 
\usepackage[english]{babel} 
\usepackage[hidelinks]{hyperref}
\usepackage[all]{xy} 
\usepackage{graphicx} 
\usepackage{chngpage} 
\usepackage{indentfirst} 
\usepackage[T1]{fontenc}
\usepackage[latin9]{inputenc}
\usepackage{prettyref} 
\usepackage{xcolor}
\usepackage[font={small,it}]{caption} 
\usepackage{sectsty} 

\setkomafont{section}{\large}
\graphicspath{{figures/}}

\numberwithin{equation}{section}
\numberwithin{figure}{section}

\theoremstyle{plain}
\newtheorem{thm}{Theorem}
\newtheorem{lem}[thm]{Lemma}
\newtheorem{prop}[thm]{Proposition}
\newtheorem{cor}[thm]{Corollary}

\theoremstyle{definition}
\newtheorem*{defn*}{Definition}
\newtheorem*{example*}{Example}

\newcommand{\bigslant}[2]{{\raisebox{.2em}{$#1$}\left/\raisebox{-.2em}{$#2$}\right.}}
\newcommand{\orcid}[1]{\href{https://orcid.org/#1}{\includegraphics[width=10pt]{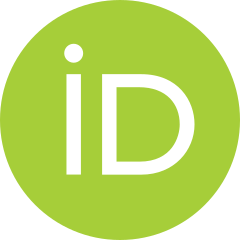}  orcid.org/#1}}

\title{\vspace{-0.5cm}\LARGE \textsc{\scalebox{0.92}[1.0]{Foundations of Structural Statistics:}\\\scalebox{0.92}[1.0]{Statistical Manifolds}}\vspace{-0.5cm}}
\author{\large P. Michl \orcid{0000-0002-6398-0654}}
\date{}

\begin{document}
\twocolumn[ 
\begin{@twocolumnfalse}  

\maketitle 

\begin{abstract}
\vspace{-1.9cm}
\begin{adjustwidth}{10mm}{10mm}
Upon a consistent topological statistical theory the application of structural statistics requires a quantification of the proximity structure of model spaces. An important tool to study these structures are Pseudo-Riemannian metrices, which in the category of statistical models are induced by statistical divergences. The present article  extends the notation of topological statistical models by a differential structure to statistical manifolds and introduces the differential geometric foundations to study  distribution families by their differential-, Riemannian- and symplectic geometry.
\end{adjustwidth}
\vspace{0.5cm}
\end{abstract}
\end{@twocolumnfalse}
] 

\section{Introduction}

Since the 1960s investigations on the invariant structures of statistical models led to various approaches to incorporate geometric structures, of which notably the contributions of \textsc{N. N. Chenzow} \cite{Chenzow1965} and \textsc{S. Amari} \cite{Amari1987} eventually encouraged the notation of \emph{Statistical Manifolds}. Thereby Amari's approach to incorporate a  differential structure emphasizes the \emph{Fisher Information Metric} which is obtained as a partial derivative of the Kullback-Leibler divergence. The induced geometry became known as \emph{Information Geometry}. Unfortunately the supplementary degree of abstraction compared to its initially low applicability caused further research on this direction to lose ``momentum''. Nevertheless, in the first decade of the $21^{\text{st}}$ century, the increasing availability of large collections of complex natural data, demanded a theoretic underpinning of complex structural assumptions, in particular with respect to the rising theory of deep-learning.

\section{\label{sec:Statistical-manifolds}Statistical Manifolds}

\emph{Topological Statistical Models} \cite{Michl2019} provide the ability to characterize statistical inference without the necessity of an underlying sample space. Thereby the statistical equivalence of statistical models is provided by Kolmogorov equivalence. The topologies of those Kolmogorov quotients in turn are obtained by countable coverings of Borel sets in $\mathbb{R}$ and therefore enriches the underlying model space to be second countable Hausdorff spaces. It is therefore straightforward to transfer the concept of topological manifolds to statistical models by the Kolmogorov quotients of their induced topological statistical models: Let $(S,\,\Sigma,\,\mathcal{M},\,\mathcal{T})$ be topological statistical model and $n\in\mathbb{N}$. Then a \emph{coordinate chart} $(U,\,\phi)$ within $\mathrm{KQ}(\mathcal{M},\,\mathcal{T})$ is constituted by an open set $U\in\bigslant{\tau}{\mathrm{id}}$ and a homeomorphism $\phi:U\rightarrow\mathbb{R}^{n}$ into $\mathbb{R}^{n}$. This allows the definition of an \emph{atlas} $\mathcal{A}$ for $\mathcal{M}$ by a family of charts $\{(U_{i},\,\phi_{i})\}_{i}$, that covers $\bigslant{\mathcal{M}}{\mathrm{id}}$, such that $\bigslant{\mathcal{M}}{\mathrm{id}}=\bigcup_{i\in I}U_{i}$. In order to extend the local Euclidean structure of the individual coordinate charts, to a global structure over the model space, the transitions within overlapping charts are required to preserve the structure. Let therefore $(U_{a},\,\phi_{a})$ and $(U_{b},\,\phi_{b})$ be coordinate charts in $\mathcal{A}$ with a nonempty intersection $U_{a\cap b}=U_{a}\cap U_{b}$, then $\phi_{a}(U_{a\cap b})$ and $\phi_{b}(U_{a\cap b})$ generally denote different representations of $U_{a\cap b}$ in $\mathbb{R}^{n}$. In this case for a given $k\in\mathbb{N}_{0}\cup\{\infty,\,\omega\}$, the charts are regarded to be $C^{k}$-compatible, iff their \emph{transition maps} $\phi_{a} \circ\phi_{b}^{-1}$ and $\phi_{b}\circ\phi_{a}^{-1}$ are $k$-times continuously differentiable.

If all charts of an atlas $\mathcal{A}$ are pairwise $C^{k}$-compatible, then $\mathcal{A}$ is termed a $C^{k}$-atlas. Let then be $\mathcal{A}^{'}$ a further $C^{k}$-atlas of $\mathcal{M}$, then $\mathcal{A}$ and $\mathcal{A}^{'}$ are termed $C^{k}$-equivalent, if also $\mathcal{A}\cup\mathcal{A}^{'}$ is a $C^{k}$-atlas of  $\mathcal{M}$. This equivalence relationship may be used to derive a maximal atlas by completion. Let therefore $\mathcal{A}_{\max}$ be the union of all $C^{k}$-atlases of $\mathcal{M}$, that are $C^{k}$-equivalent to $\mathcal{A}$, then $\mathcal{A}_{\max}$ is unique for the $C^{k}$-equivalence class of $\mathcal{A}$ and does not depend on the choice of $\mathcal{A}$ within this class. Then any $C^{k}$-differentiable function, that is defined within the image of a chart in $\mathcal{A}_{\max}$ has a unique $C^{k}$-differentiable extension within its neighbourhood in $\mathrm{KQ}(\mathcal{M},\,\mathcal{T})$.
The crux in the definition of a $C^{k}$-atlas $\mathcal{A}$ however
is, that due to the Hausdorff property of $\mathrm{KQ}(\mathcal{M},\,\mathcal{T})$
and the completion of $\mathcal{A}$ by$\mathcal{A}_{\max}$ the requirement of the transition functions to be $C^{k}$-diffeomorphism in $\mathbb{R}^{n}$ induces a differential structure to $\mathrm{KQ}(\mathcal{M},\,\mathcal{T})$. Then not only the transition maps, but any coordinate chart by itself may be regarded as a $C^{k}$-diffeomorphism into $\mathbb{R}^{n}$. This defines the structure of a \emph{statistical manifold}.
\begin{defn*}[Statistical manifold]
\label{def:Statistical-manifold} \emph{Let $(S,\,\Sigma,\,\mathcal{M},\,\mathcal{T})$
be a topological statistical model and $\mathcal{A}$ an $n$-dimensional
$C^{k}$-atlas for $\mathrm{KQ}(\mathcal{M},\,\mathcal{T})$. Then
the tuple $(S,\,\Sigma,\,\mathcal{M},\,\mathcal{A})$ is termed a
statistical manifold. }\textbf{Remark}:\emph{ The category of $k$-differentiable
statistical manifolds is denoted by $\mathbf{StatMan}^{k}$.}
\end{defn*}
Since the atlas $\mathcal{A}$ has to be defined with regard to the
Kolmogorov quotient $\mathrm{KQ}(\mathcal{M},\,\mathcal{T})$ to assure
the Hausdorff property, statistical manifolds have technically to
be regarded as non-Hausdorff manifolds. Since the atlas $\mathcal{A}$
conversely induces a topology that equals $\bigslant{\mathcal{T}}{\mathrm{id}}$
the original topological statistical model $(S,\,\Sigma,\,\mathcal{M},\,\mathcal{T})$
may not be derived by \emph{$(S,\,\Sigma,\,\mathcal{M},\,\mathcal{A})$}.
Nevertheless by the extension of the Kolmogorov quotient to the atlas $\mathrm{KQ}(\mathcal{M},\,\mathcal{A})$, it follows that $\mathrm{KQ}(\mathcal{M},\,\mathcal{T})$ and $\mathrm{KQ}(\mathcal{M},\,\mathcal{A})$ are Kolmogorov equivalent and therefore that $(S,\,\Sigma,\,\mathcal{M},\,\mathcal{T})$ and $(S,\,\Sigma,\,\mathcal{M},\,\mathcal{A})$ are induced by statistical equivalent models. With regard to observation based statistical inference this ``irregularity'' however usually has no impact, since for any \emph{identifiable statistical manifold} $(S,\,\Sigma,\,\mathcal{M},\,\mathcal{A})$, which model space $\mathcal{M}$ is identical to a parametric family $\mathcal{M}_{\theta}$, it holds that $\bigslant{\mathcal{M}}{\mathrm{id}}=\mathcal{M}_{\theta}=\mathcal{M}$ and therefore that $\mathrm{KQ}(\mathcal{M},\,\mathcal{A})=(\mathcal{M},\,\mathcal{A})$.

Without loss of generality in the following therefore $(S,\,\Sigma,\,\mathcal{M},\,\mathcal{A})$
is assumed to be an identifiable statistical manifold and therefore
a manifold in the usual context. Nevertheless, in order to provide
higher structures, it is reasonable to recapitulate the usual concepts
and the vocabulary of manifolds. First of all by assuming $\mathcal{A}$
to be a $C^{0}$-atlas, the transition maps, that define the structure
of $(S,\,\Sigma,\,\mathcal{M},\,\mathcal{A})$ are only required to
be continuous and $(S,\,\Sigma,\,\mathcal{M},\,\mathcal{A})$ is termed
\emph{topological}. For the case, that $\mathcal{A}$ is a $C^{k}$-atlas
with $k>0$ however, the transition functions at least have to be
differentiable and therefore $(S,\,\Sigma,\,\mathcal{M},\,\mathcal{A})$
is termed \emph{differentiable}. Let now be $(\mathcal{N},\,\mathcal{B})$
a further identifiable statistical manifold, where $\mathcal{B}$
is an $m$-dimensional $C^{k}$-atlas of $\mathcal{N}$ and $f\colon\mathcal{M}\to\mathcal{N}$
a function, that is is continuous w.r.t. the induced topologies. Then
$f$ is \emph{$C^{k}$-differentiable} and written as $f\in C^{k}(\mathcal{M},\,\mathcal{N})$,
if for arbitrary coordinate charts $(U,\,\phi)\in\mathcal{A}$ and
$(V,\,\nu)\in\mathcal{B}$ with $f(U)\subseteq V$ it holds, that
$\nu\circ f\circ\phi^{-1}$ is $k$-times continuously differentiable.
Thereby the case $(\mathcal{N},\,\mathcal{B})=(\mathbb{R},\,\mathcal{B}(\mathbb{R}))$
occupies an exceptional position, which is abbreviated by the notation
$C^{k}(\mathcal{M})$. At this point it is important to notice, that
although the definitions of $C^{k}$-atlases and $C^{k}$-differentiable
functions depend on the choice of $k$, this does essentially not
apply for the underlying differentiable structures. The reason for
this ``peculiarity'' may be found in the property, that for any
$k>0$ any $C^{k}$-atlas uniquely admits a ``smoothing'', given
by a $C^{k}$-equivalent $C^{\infty}$-atlas. Therefore the set of
smooth functions $C^{\infty}(\mathcal{M})$ is well defined, independent
of the underlying differentiable structure. Thereby $C^{\infty}(\mathcal{M})$
constitutes an associative algebra w.r.t. the pointwise product ``$\cdot$'',
the addition ``$+$'' and the scalar multiplication. This allows
a formal definition of \emph{derivations} at points $P\in\mathcal{M}$
by $\mathbb{R}$-linear functions $D\colon C^{\infty}(\mathcal{M})\to\mathbb{R}$,
that satisfies the Leibniz rule $D(\varphi\cdot\psi)=D(\varphi)\psi(P)+\varphi(P)D(\psi)$
for all $\varphi,\,\psi\in C^{\infty}(\mathcal{M})$. Let now be $T_{P}\mathcal{M}$
the set of all derivations at $P$, then $T_{P}\mathcal{M}$ defines
an $n$-dimensional $\mathbb{R}$-vector space, by the operations
$(v+w)(\varphi)\coloneqq v(\varphi)+w(\varphi)$ and $(\lambda v)(\varphi)\coloneqq\lambda v(\varphi)$,
for $v,\,w\in T_{P}\mathcal{M}$, $\varphi\in C^{\infty}(\mathcal{M})$
and $\lambda\in\mathbb{R}$. As for any given coordinate chart $(U,\,\phi)$
that contains $P$, any derivation $v\in T_{P}\mathcal{M}$ uniquely
corresponds to a directional derivative in $\mathbb{R}^{n}$ at the
point $\phi(P)$, the elements of $T_{P}\mathcal{M}$ are termed\emph{
tangent vectors} and $T_{P}\mathcal{M}$ the \emph{tangent space}
at $P$. Then the\emph{ partial derivatives }at $P$, given by $\{\partial_{i}\}_{P}$
with $\partial_{i}\colon P\mapsto\partial/\partial\phi^{i}\mid_{P}$
provide a basis of $T_{P}\mathcal{M}$, such that any $v\in T_{P}\mathcal{M}$
has a local \emph{representation} by a vector $\boldsymbol{\xi}\in\mathbb{R}^{n}$
with $v=\xi^{i}\partial_{i}$. Let now be $f\in C^{\infty}(\mathcal{M},\,\mathcal{N})$,
then the \emph{differential} of $f$ at $P\in\mathcal{M}$ is a linear
mapping $\mathrm{d}f_{P}\colon T_{P}\mathcal{M}\to T_{f(P)}\mathcal{N}$,
which for all $v\in T_{P}\mathcal{M}$ and $\varphi\in C^{\infty}(\mathcal{N})$
is defined by $(\mathrm{d}f_{P}v)(\varphi)\coloneqq v(\varphi\circ f)$.
Then $f$ is an \emph{immersion}, if for all $P\in\mathcal{M}$ the
differential $\mathrm{d}f_{P}$ is injective. If furthermore $f$
is injective and continuous w.r.t. to the respectively induced topologies,
then $f$ is a \emph{smooth embedding} and the image of $f$ a \emph{smooth
submanifold} of $\mathcal{N}$ w.r.t. the atlas, which is restricted
to the image. This allows the definition of \emph{smooth parametrisation}s
for statistical manifolds. 
\begin{defn*}[Smooth parametrisation]
\label{def:Differentiable-parametrisation} \emph{Let $(\mathcal{M},\,\mathcal{A})$
be a differentiable statistical manifold, $(V,\,\mathcal{B})$ a differentiable
manifold over a vector space $V$ and $\theta$ a parametrisation
for $\mathcal{M}$ over $V$. Then $\theta$ is termed a smooth parametrisation
for $(\mathcal{M},\,\mathcal{A})$, iff $\theta^{-1}\colon\mathrm{KQ}(\mathcal{M},\,\mathcal{A})\hookrightarrow(V,\,\mathcal{B})$
is a smooth embedding.}
\end{defn*}
Since for any smooth $n$-manifold the \emph{Whitney embedding theorem}
postulates the existence of a smooth embedding within $\mathbb{R}^{2n}$
any smooth statistical $n$-manifold \emph{$(\mathcal{M},\,\mathcal{A})$}
has a smooth parametrisation $\theta$ over $\mathbb{R}^{2n}$. It
is therefore convenient to introduce the notation ``$(\mathcal{M}_{\theta},\,\mathcal{A})$''
for a differentiable statistical manifold with a smooth parametrisation
$\theta$. Since the tuple $(\mathcal{M},\,\theta^{-1})$ is an $C^{\infty}$-chart
that covers $\mathcal{M}$, it provides a \emph{smooth representation}
of \emph{$(\mathcal{M},\,\mathcal{A})$} by the parameter space $\Theta=\theta^{-1}(\mathcal{M})$.
Thereby for any coordinate chart $(U,\,\phi)\in\mathcal{A}$, the
mapping $\phi\colon U\to\mathbb{R}^{n}$ provides an $n$-dimensional
basis for the tangent space $T_{P}\mathcal{M}$ by the partial derivatives
$\{\partial_{i}\}_{P}$. Since the smooth parametrisation $\theta$
however is an immersion, also the differentials $\mathrm{d}\theta^{-1}(\partial_{i})$
also provide an $n$-dimensional basis of the space $T_{\theta^{-1}(P)}\Theta$.
Therefore any tangent space may intuitively be identified with an
$n$-dimensional affine subspace of $\Theta$ and any tangent vector
by a directional derivative within this subspace. A smooth parametrisation
then in particular allows the identification of \emph{smooth curves}
$\gamma\colon I\to\mathcal{M}$ in $\mathcal{M}$ by smooth \emph{parametric
curves} $\gamma_{\theta}\colon I\to\Theta$ in $\Theta$, such that
$\gamma=(\theta\circ\gamma_{\theta})$. Then for any $k\in\mathbb{N}_{0}\cup\{\infty,\,\omega\}$,
it holds that $\gamma\in C^{k}(I,\,\mathcal{M})$ iff $\gamma_{\theta}\in C^{k}(I,\,\Theta)$.
Therefore smooth parametric curves provide the foundation for a traversal
of $\mathcal{M}$. Thereby the traversal along the parametric curve
$\gamma_{\theta}(t)$ is described by the unique directional derivatives
$\dot{\gamma}_{\theta}(t)$ in $\mathcal{\Theta}$. Consequentially
due to the unique correspondence between directional derivatives in
$T_{\gamma_{\theta}(t)}\mathcal{\Theta}$ and tangent vectors in $T_{\gamma(t)}\mathcal{M}$
the traversal along $\gamma$ in $\mathcal{M}$ is also described
by unique tangent vectors $\dot{\gamma}(t)\in T_{\gamma(t)}\mathcal{M}$.
This uniqueness however only applies to the direction in $\Theta$
but not to its representation in $T_{\gamma_{\theta}(t)}\mathcal{\Theta}$
in terms of the chosen basis. With regard to a coordinate chart $(U,\,\phi)$
in $\mathcal{M}$ the local basis $\{\partial_{i}\}_{P}$ at $P\in U$
naturally extends over $U$, by regarding $\partial_{i}\colon P\mapsto\partial/\partial\phi^{i}$
as an ordered basis, termed a \emph{local frame}, which localized
as $P$ provides $\partial_{i}\colon P\mapsto\partial/\partial\phi^{i}\mid_{P}$.
Then the differentials $\mathrm{d}\theta^{-1}(\partial_{i})$ provide
a local basis of $T_{\gamma_{\theta}(t)}\Theta\mid_{\theta^{-1}(U)}$
and the directional derivatives $\dot{\gamma}_{\theta}(t)\in T_{\gamma_{\theta}(t)}\Theta\mid_{\theta^{-1}(U)}$
may uniquely be identified with tangent vectors $\dot{\gamma}(t)\in T_{\gamma(t)}\mathcal{M}\mid_{U}$
by $\dot{\gamma}(t)=\mathrm{d}\theta(\dot{\gamma}_{\theta}(t))\mid_{U}$.
In order to continue the traversal however it is required to ``connect''
the basis vectors of the affine spaces along the curve by an unambiguous
notation, which is independent of the chosen coordinate charts. This
provides the notation of an \emph{affine connection}. At a global
scale the disjoint union of all tangent spaces constitutes the \emph{tangent
bundle} $T\mathcal{M}$, which by itself is diffeomorphic to $\mathcal{M}\times\mathbb{R}^{n}$
and therefore in particular a differentiable manifold. This property
allows to define \emph{smooth vector fields} on $\mathcal{M}$ by
smooth functions $X\in C^{\infty}(\mathcal{M},\,T\mathcal{M})$, or
w.r.t. the sequence $\mathcal{M}\stackrel{X}{\hookrightarrow}T\mathcal{M}\twoheadrightarrow\mathcal{M}$
by smooth sections $X\in\Gamma(T\mathcal{M})$. Intuitively smooth
vector fields assign tangent vectors to the points of the manifold,
such that ``small `` movements on the manifold are accompanied by
``small'' changes within the tangent spaces. With regard to a coordinate
chart $(U,\,\phi)$ a local frame $\{\partial_{i}\}$ may also be
regarded as a localized ordered basis of the vector fields $\Gamma(T\mathcal{M}\mid_{U})$.
Therefore the transition of local frames may be described be derivatives
of vector fields, which provides the notation of \emph{covariant derivatives}.
A covariant derivative $\nabla$ on $\mathcal{M}$ formally defines
a mapping $\nabla\colon\Gamma(T\mathcal{M})^{2}\to\Gamma(T\mathcal{M})$
with $(X,\,Y)\mapsto\nabla_{X}Y$, which satisfies: (i) $\nabla$
is $\mathbb{R}$-linear in both arguments, (ii) $\nabla$ is $C^{\infty}(\mathcal{M})$-linear
in the first argument and (iii) $\nabla$ is a derivation in the second
argument, such that $\nabla_{X}(\varphi\cdot Y)=X(\varphi)Y+\varphi\nabla_{X}Y$
for arbitrary $\varphi\in C^{\infty}(\mathcal{M})$ and $X,\,Y\in\Gamma(T\mathcal{M})$.
An affine connection is then completely described by the specification
of a covariant derivative which in turn endows a differentiable manifold
with an additional structure $\nabla$. In particular however the
choice of an affine connection for any curve $\gamma$ completely
determines its derivative $\dot{\gamma}\in\Gamma(T\mathcal{M})\mid_{\gamma}$
along the curve, as well as those vector fields $X\in\Gamma(T\mathcal{M})$
which are covariant constant along $\gamma$, such that $\nabla_{\dot{\gamma}}X=0$.
As this property however may also be applied w.r.t. the derivative
along the curve itself, the choice of an affine connection $\nabla$
in particular determines those curves $\gamma$, which derivative
is covariant constant along their traversal. These curves, known as
\emph{geodesic}s, therefore generalize straight lines to differentiable
manifolds.
\begin{defn*}[Geodesic]
\label{def:Geodesic}\emph{ Let $(\mathcal{M},\,\mathcal{A})$ be
a smooth statistical manifold and $\nabla$ an affine connection on
$\mathrm{KQ}(\mathcal{M},\,\mathcal{A})$. Then a smooth curve $\gamma\colon I\to\mathcal{M}$
is termed a geodesic w.r.t. $\nabla$, iff it satisfies the geodesic
equation: 
\[
\nabla_{\dot{\gamma}}\dot{\gamma}=0
\]
}
\end{defn*}
Over and above geodesic, the choice of an affine connection $\nabla$
admits two fundamental invariants to the differentiable structure
by the \emph{curvature} and the \emph{torsion. }Thereby the curvature
$R\colon\Gamma(T\mathcal{M})^{3}\to\Gamma(T\mathcal{M})$ with $R(X,\,Y)Z\coloneqq\nabla_{X}\nabla_{Y}Z-\nabla_{Y}\nabla_{X}Z-\nabla_{[X,\,Y]}Z$
intuitively provides a description, of how tangent spaces ``roll''
along smooth curves under parallel transport, whereas the torsion
$T\colon\Gamma(T\mathcal{M})^{2}\to\Gamma(T\mathcal{M})$ with $T(X,\,Y)\coloneqq\nabla_{X}Y-\nabla_{Y}X-[X,\,Y]$
describes their ``twist'' about the curve. Notwithstanding these
invariants however, an affine connection $\nabla$ only adjusts local
tangent spaces, but does not provide a notation of ``length'' or
``angle''. The mandatory next step, therefore regards the incorporation
of local geometries within the tangent spaces, that eventually extend
to a global geometry over the differentiable structure.

\section{Pseudo-Riemannian Structure}

Since tangent spaces are vector spaces, it is natural to obtain the
local geometry by an inner product. More generally however it suffices
to provide a\emph{ }mapping $g_{P}\colon T_{P}\mathcal{M}{}^{2}\to\mathbb{R}$
that satisfies (i) $g_{P}$ is $C^{\infty}(\mathcal{M})$-bilinear,
(ii) $g_{P}$ is symmetric and (iii) $g_{P}$ is non-degenerate. In
the purpose to extend the local geometries to a global geometry however,
it has additionally to be claimed that the local geometries only vary
smoothly w.r.t. smooth vector fields. This localization requirement
yields the notation of a \emph{pseudo-Riemannian metric $g$} on $(\mathcal{M},\,\mathcal{A})$,
which endows each point $P\in\mathcal{M}$ with a symmetric non-degenerate
form $g_{P}$, such that the mapping 
\[
g(X,\,Y)\colon P\mapsto g_{P}(X_{P},\,Y_{P})
\]
is smooth, i.e. $g(X,\,Y)\in C^{\infty}(\mathcal{M})$, for arbitrary
$X,\,Y\in\Gamma(T\mathcal{M})$. With regard to a coordinate chart
$(U,\,\phi)$ and a local frame $\{\partial_{i}\}$ the pseudo-Riemannian
metric $g$ has a coordinate representation by \emph{metric coefficients}
$g_{ij}\colon U\to\mathbb{R}$, with $g_{ij}=g(\partial_{i},\,\partial_{j})$
and therefore by a matrix $G_{P}=(g_{ij})$, termed a \emph{fundamental
matrix}. For $P\in U$ and $v,\,w\in T_{P}\mathcal{M}$, with $v=\xi^{i}\partial_{i}$
and $w=\zeta^{i}\partial_{i}$ it then follows, that $g(v,\,w)$ has
a local representation $\langle\boldsymbol{\xi},\,\boldsymbol{\zeta}\rangle_{P}\coloneqq\boldsymbol{\xi}^{T}G_{P}\boldsymbol{\zeta}$.
In the purpose to extend the local geometry to a global geometry an
affine connection $\nabla$ has to be defined, which is compatible
with $g$, such that $Xg(Y,\,Z)=g(\nabla_{X}Y,\,Z)+g(Y,\,\nabla_{X}Z)$,
for all $X,\,Y,\,Z\in\Gamma(T\mathcal{M})$. In this case $\nabla$
is termed a \emph{metric connection} and has a coordinate representation
by \emph{connection coefficients} $\Gamma_{ij}^{k}\colon U\to\mathbb{R}$,
with $\nabla_{\partial_{i}}\partial_{j}=\sum_{k=1}^{n}\Gamma_{ij}^{k}\partial_{k}$,
known as the \emph{Christoffel-symbols. }Then the geodesic equation
over a parametric curve $\gamma_{\theta}$ is locally expressed by
a second order ODE:

\begin{equation}
\nabla_{\dot{\gamma}_{\theta}}\dot{\gamma}_{\theta}=0\Longleftrightarrow\ddot{\gamma}_{\theta}^{k}+\sum_{i,j}\Gamma_{ij}^{k}\dot{\gamma}_{\theta}^{i}\dot{\gamma}_{\theta}^{j}=0,\,\forall k\label{eq:geodesic-equation}
\end{equation}
With little effort, the \emph{Picard-Lindel�f theorem} then assures,
that for any $(P,\,v)\in T\mathcal{M}$ there exists a locally unique
geodesic $\gamma_{P,v}\colon I\to\mathcal{M}$, that satisfies the
initial conditions $\gamma_{P,v}(0)=P$ and $\dot{\gamma}_{P,v}(0)=v$.
Thereby the locally uniqueness extends to an maximal open interval
$I=(a,\,b)$ in $\mathbb{R}$. If $\nabla$ is furthermore \emph{torsion
free} i.e. $T(X,\,Y)=0$, for all $X,\,Y\in\Gamma(T\mathcal{M})$,
then $\nabla$ is termed a \emph{Levi-Civita connection} and the Christoffel-symbols
may explicitly be derived by the equation $\Gamma_{ij}^{k}=\frac{1}{2}\sum_{l}g^{kl}(\partial_{i}g_{jl}+\partial_{j}g_{il}-\partial_{l}g_{ij})$,
where $g^{kl}$ denote the coefficients of the inverse fundamental
matrix $G_{P}^{-1}$, which existence is assured by the properties
of $g_{P}$. Therefore it follows, that any pseudo-Riemannian metric
$g$ uniquely admits a Levi-Civita connection $\nabla^{g}$. For this
reason the choice of a pseudo-Riemannian metric naturally induces
a global geometry to a differentiable manifold and therefore w.r.t.
statistical manifolds justifies the definition of \emph{pseudo-Riemannian
statistical manifolds}.
\begin{defn*}[Pseudo-Riemannian statistical manifold]
\label{def:Riemannian-statistical-manifold} \emph{Let $(\mathcal{M},\,\mathcal{A})$
be a differentiable statistical manifold and $g$ a Pseudo-Riemannian
metric on $\mathrm{KQ}(\mathcal{M},\,\mathcal{A})$. Then the tuple
$(\mathcal{M},\,g)$ is termed a Pseudo-Riemannian statistical manifold.
}\textbf{Remark}:\emph{ The category of $k$-differentiable Pseudo-Riemannia
statistical manifolds is denoted by $\mathbf{StatMan}_{\text{R}}^{k}$.}
\end{defn*}
Generally Pseudo-Riemannian manifolds endow the notation of geodesics
with an intuitive meaning as the trajectories of free particles. Thereby
the equations of motion obey the principle of stationary action, whereat
the \emph{action functional} $\mathcal{S}(\gamma)\coloneqq\int_{a}^{b}\mathcal{L}\mathrm{d}t$
is defined over the Lagrangian\emph{ }$\mathcal{L}\coloneqq\frac{1}{2}g(\dot{\gamma},\,\dot{\gamma})$.
The properties of the Levi-Civita connection then allow the transformation
of the local geodesic equation \ref{eq:geodesic-equation} to \emph{Euler-Lagrange
equations}, over the local Lagrangian $\mathcal{L}_{\theta}\coloneqq\sum_{i,j}\frac{1}{2}g_{ij}\dot{\gamma}_{\theta}^{i}\dot{\gamma}_{\theta}^{j}$,
such that:
\begin{equation}
\nabla_{\dot{\gamma}_{\phi}}^{g}\dot{\gamma}_{\phi}=0\Longleftrightarrow\frac{\mathrm{d}}{\mathrm{d}t}\left(\frac{\partial\mathcal{L}_{\phi}}{\partial\dot{\gamma}_{\phi}^{k}}\right)-\frac{\partial\mathcal{L}_{\phi}}{\partial\gamma_{\phi}^{k}}=0,\,\forall k\label{eq:Euler-Lagrange-equations}
\end{equation}
Consequentially the geodesics of Pseudo-Riemannian manifolds are
stationary solutions of the action functional $\mathcal{S}(\gamma)$,
i.e. $\delta\mathcal{S}=0$. By regarding the tangent bundle $T\mathcal{M}$
as the \emph{configuration space} of a moving particle and its elements
$(q,\,\dot{q})\in T\mathcal{M}$ as the\emph{ generalized coordinates},
the Lagrangian equals its \emph{kinetic term}. With regard to equation
\ref{eq:Euler-Lagrange-equations} the geodesics of $(\mathcal{M},\,g)$
then coincide with the trajectories of free particles. This encourages
the interpretation of $(\mathcal{M},\,g)$ as a dynamical system,
where the temporal evolution is determined by the \emph{geodesic flow
}$\Phi^{t}:T\mathcal{M}\to T\mathcal{M}$ with $\Phi^{t}(q,\,\dot{q})=(\gamma_{q,\dot{q}}(t),\,\dot{\gamma}_{q,\dot{q}}(t))$,
where $\gamma_{q,\dot{q}}(t)$ is the locally unique geodesic, that
satisfies the initial conditions $\gamma_{q,\dot{q}}(0)=q$ and $\dot{\gamma}_{q,\dot{q}}(0)=\dot{q}$.\emph{
}Then due to $\frac{\mathrm{d}}{\mathrm{d}t}g(\dot{q},\,\dot{q})=g(\nabla_{\dot{q}}^{g}\dot{q},\,\dot{q})=0$
the geodesic flow preserves the kinetic term along its trajectories,
and therefore generalizes Newton's first law of motion to curvilinear
and pseudo-Euclidean spaces. In appreciation of its origins in the
conceptualization of spacetime, a geodesic $\gamma$ is therefore
termed \emph{spacelike} if $g(\dot{q},\,\dot{q})>0$, \emph{lightlike}
if $g(\dot{q},\,\dot{q})=0$ and \emph{timelike} if $g(\dot{q},\,\dot{q})<0$.
Moreover the Pseudo-Riemannian metric $g$ induces a canonical isomorphism
between the tangent spaces $T_{q}\mathcal{M}$ and their respective
dual spaces $T_{q}^{*}\mathcal{M}$, the \emph{cotangent spaces},
which assigns a \emph{cotangent vector} $p\in T_{q}^{*}\mathcal{M}$
to each tangent vector $\dot{q}\in T_{q}\mathcal{M}$ by $p(v)\coloneqq g_{q}(\dot{q},\,v)$.
Then also the choice of a local frame $\{\partial_{i}\}$ uniquely
induces a \emph{local coframe} $\{\mathrm{d}q^{i}\}$ by $\mathrm{d}q^{i}\coloneqq\partial_{i}^{T}G_{q}$,
such that any $p\in T_{q}^{*}\mathcal{M}$ has a local representation
$p=p_{i}\mathrm{d}q^{i}$. As the geodesic flow however preserves
the kinetic term it holds, that $\frac{\mathrm{d}}{\mathrm{d}t}p(\dot{q})=\frac{\mathrm{d}}{\mathrm{d}t}g(\dot{q},\,\dot{q})=0$,
such that $p$ equals the \emph{conjugate momentum} of $\dot{q}$.
Consequentially the disjoint union of all cotangent spaces, given
by the\emph{ cotangent bundle} $T^{*}\mathcal{M}$ equals the \emph{phase
space} of the dynamical system. Finally by the definition of the \emph{Hamiltonian}
$\mathcal{H}(q,\,p)\coloneqq\frac{1}{2}g^{ij}p_{i}p_{j}$ with $(g^{ij})=G_{q}^{-1}$
it follows, that $(\mathcal{M},\,g)$ uniquely corresponds to a \emph{Hamiltonian
system}, since (i) $\dot{q}^{i}=g^{ij}p_{j}=\frac{\partial\mathcal{H}}{\partial p_{i}}$
and (ii) $\dot{p}^{i}=-\frac{\partial}{\partial\gamma_{i}}\frac{1}{2}g^{ij}p_{i}p_{j}=-\frac{\partial\mathcal{H}}{\partial q_{i}}$.
This representation allows to reformulate the principle of stationary
action in \emph{canonical coordinates} $(q,\,p)\in T^{*}\mathcal{M}$
by the \emph{curve integral }$\mathcal{S}(q)=\int_{a}^{b}\mathcal{L}\mathrm{d}t=\frac{1}{2}\int_{q}p$.
In particular this formulation, known as \emph{Maupertuis' principle}
then describes the trajectory of a free particle by its geometric
shape instead of its temporal evolution. Due to this geometric interpretation
the action functional of a curve may also be defined with regard to
a given vector field of conjugate momenta $p\in\Gamma(T^{*}\mathcal{M})$,
by $\mathcal{S}_{p}(q)\coloneqq\frac{1}{2}\int_{q}p$. Then for arbitrary
smooth curves $q\colon[a,\,b]\to\mathcal{M}$ the action $\mathcal{S}_{p}(q)$
is completely determined by its boundary values localized at $q_{a}$
and $q_{b}$, such that: 

\[
\mathcal{S}_{p}(q)=\frac{1}{2}\int_{q}p=\mathcal{S}_{p}(q_{b})-\mathcal{S}_{p}(q_{a})
\]
Although w.r.t. the fundamental theorem of calculus this insight seems
rather trite, it provides a far-reaching generalisation. Thereby in
the very same manner as the smooth sections of $T^{*}\mathcal{M}$
constitute the smooth linear forms over $T\mathcal{M}$, the smooth
alternating multilinear forms over $T\mathcal{M}^{k}$, termed \emph{differential
$k$-forms} are given by smooth sections of the \emph{outer product}
$\Lambda^{k}(T^{*}\mathcal{M})$. These $k$-forms then provide the
natural integrands over curves, surfaces, volumes or higher-dimensional
$k$-manifolds and therefore may be thought as measures of the flux
through infinitesimal $k$-parallelotopes. In this sense the smooth
functions over $\mathcal{M}$ are \emph{$0$-forms} and the smooth
vector fields over $\mathcal{M}$ are \emph{$1$-forms}. In particular
however since for any smooth function $f\in C^{\infty}(\mathcal{M})$
the differential $\mathrm{d}f$ is a smooth vector field, it appears
that $\mathrm{d}$ by itself defines an $\mathbb{R}$-linear mapping
$\mathrm{d}\colon\Omega^{0}(\mathcal{M})\to\Omega^{1}(\mathcal{M})$,
where $\Omega^{k}(\mathcal{M})\coloneqq\Gamma(\Lambda^{k}(T^{*}\mathcal{M}))$.
This encourages to extend the notation of a differential to arbitrary
$k$-forms by the \emph{exterior derivative}, given by an $\mathbb{R}$-linear
mapping $\mathrm{d}_{k}\colon\Omega^{k}(\mathcal{M})\to\Omega^{k+1}(\mathcal{M})$,
that satisfies (i) $\mathrm{d}_{k}$ is an antiderivation for any
$k\in\mathbb{N}_{0}$, (ii) $\mathrm{d}_{k+1}\circ\mathrm{d}_{k}=0$
for any $k\in\mathbb{N}_{0}$ and (iii) $\mathrm{d}_{0}$ is the differential.
In more detail (i) claims, that for any $\alpha\in\Omega^{k}(\mathcal{M})$
and $\beta\in\Omega^{l}(\mathcal{M})$ it follows, that $\mathrm{d}_{k+l}(\alpha\wedge\beta)=\mathrm{d}_{k}\alpha\wedge\beta+(-1)^{l}(\alpha\wedge\mathrm{d}_{l}\beta)$.
This provides the property, that infinitesimal changes of the volume
$\alpha\wedge\beta$ are expressible as the sum of infinitesimal changes
in their orthocomplemented constituent volumes. Then the additional
claim (ii) assures the symmetry of second derivatives and (iii) the
compatibility with the differential. In order to provide a measure
of length however the Pseudo-Riemannian metric $g$ is additionally
required to be positive definite, i.e. such that $\forall P\in\mathcal{M}$
the $g_{p}$ are positive definite. Then $g$ is termed a \emph{Riemannian
metric} and a statistical manifold $(\mathcal{M},\,g)$ a \emph{Riemannian
statistical manifold}. In this case the Riemannian metric defines
an inner product $\langle\cdot,\,\cdot\rangle_{g}\colon T_{P}\mathcal{M}{}^{2}\to\mathbb{R}$
by $(v,\,w)\mapsto g_{P}(v,\,w)$ and therefore induces a norm $\|\cdot\|_{g}\colon T_{P}\mathcal{M}\to\mathbb{R}$
by $\|v\|_{g}\coloneqq\sqrt{\langle v,\,v\rangle_{g}}$. This allows
the definition of the \emph{length} functional of a piecewise smooth
curve.
\begin{defn*}[Arc length]
\label{def:Length-of-a-curve} \emph{Let $(\mathcal{M},\,g)$ be
a Riemannian statistical manifold and $\gamma\colon[a,\,b]\to\mathcal{M}$
a piecewise smooth curve in $\mathrm{KQ}(\mathcal{M},\,g)$. Then
the arc length of $\gamma$ is given by:
\begin{equation}
L(\gamma)\coloneqq\int_{a}^{b}\|\dot{\gamma}(t)\|_{g}\mathrm{d}t\label{eq:manifold:metric:geodesic}
\end{equation}
}
\end{defn*}
Analogues to the action functional, the length functional may be written
by a Lagrangian, which is given by $\mathcal{L}_{L}(\gamma,\,\dot{\gamma},\,t)\coloneqq\sqrt{g(\dot{\gamma},\,\dot{\gamma})}$,
such that $\mathcal{L}_{L}=\sqrt{2\mathcal{L}}$. Then the Euler-Lagrange
equations for the length functional are equivalent to the Euler-Lagrange
equations for action functional, such that the stationary solutions
of the length and action functional coincide. This property allows
to equip Riemannian statistical manifolds with a \emph{distance.}
\begin{defn*}[Distance]
\label{def:Distance} \emph{Let $(\mathcal{M},\,g)$ be a Riemannian
statistical manifold, then the distance }$d\colon\mathcal{M}^{2}\to\mathbb{R}$\emph{
of $P,\,Q\in\mathcal{M}$ is defined by:
\begin{equation}
d(P,\,Q)\coloneqq\begin{cases}
\infty & \text{, if \ensuremath{P} and \ensuremath{Q} are not}\\
 & \text{ path connected in \ensuremath{\mathcal{M}}}\\
\inf L(\gamma) & \text{, where \ensuremath{\gamma\colon[a,\,b]\to\mathcal{M}} }\\
 & \text{ with \ensuremath{\gamma(a)=P,\,\gamma(b)=Q}}
\end{cases}\label{eq:manifold:metric:distance-1}
\end{equation}
}
\end{defn*}
Due to its definition the distance $d$ of a Riemannian statistical
manifold for arbitrary $P,\,Q,\,R\in\mathcal{M}$ satisfies, that:
(i) $d(P,\,P)=0$, (ii) $d(P,\,Q)=d(Q,\,R)$ and (iii) $d(P,\,Q)+d(Q,\,R)\leq d(P,\,R)$.
In order to show, that $(\mathcal{M},\,d)$ is a metric space it therefore
suffices to prove, that $d(P,\,Q)>0$ for $P\ne Q$. Let $(\mathcal{M},\,g)$
and $(\mathcal{N},\,g^{'})$ be Pseudo-Riemannian manifolds, and $f\in C^{\infty}(\mathcal{M},\,\mathcal{N})$,
then $f$ is an isometry, iff $g_{P}(v,\,w)=g_{f(P)}^{'}(\mathrm{d}f_{P}(v),\,\mathrm{d}f_{w}(P))$
for all $P\in\mathcal{M}$ and $v,\,w\in T_{P}\mathcal{M}$ .

Let $(\mathcal{P},\,\mathcal{A})$ be a differentiable statistical
manifold. \emph{Then a mapping $D(\cdot\parallel\cdot)\colon\mathcal{P}\times\mathcal{P}\to\mathbb{R}^{+}$
is termed a divergence over $\mathcal{P}$, iff for all $P,\,Q\in\mathcal{P}$
it holds, that $D(P\parallel Q)\geq0$, with $D(P\parallel Q)=0\Leftrightarrow P=Q$}.
If 

these scalar products are usually induced by derivatives of locally
linear divergences.
\begin{defn*}[Statistical divergence]
\label{def:Divergence} \emph{Let 
\[
(X,\,\mathcal{P})\in\mathrm{ob}(\mathbf{Stat})
\]
 Then a mapping $D(\cdot\parallel\cdot)\colon\mathcal{P}\times\mathcal{P}\to\mathbb{R}^{+}$
is termed a (statistical) divergence over $\mathcal{P}$, iff for
all $P,\,Q\in\mathcal{P}$ it holds, that 
\begin{align}
 & D(P\parallel Q)\geq0,\,\forall P,\,Q\in\mathcal{P}\\
 & D(P\parallel Q)=0\Leftrightarrow P=Q,\,\forall P,\,Q\in\mathcal{P}
\end{align}
}
\end{defn*}
\begin{defn*}[Locally linear divergence]
\label{def:Locally-linear-divergence}\emph{ Let 
\[
(X,\,\mathcal{P_{\xi}})\in\mathrm{ob}(\mathbf{StatMan}^{k})
\]
 and let $D$ be a statistical divergence over $(X,\,\mathcal{P})$.
Then $D$ is termed locally linear, iff for all $P\in\mathcal{P}$
the linearisation of $D$ at $P$ is given by a positive definite
matrix $G_{\xi}(P)$ , such that:
\begin{equation}
D[P_{\xi}\parallel P_{\xi}+\mathrm{d}P]=\frac{1}{2}\mathrm{d}\boldsymbol{\xi}^{T}G_{\xi}(P_{\xi})\mathrm{d}\boldsymbol{\xi}+O(n^{3})\label{eq:divergence_taylor}
\end{equation}
}
\end{defn*}
\begin{example*}[Kullback-Leibler divergence]
\label{exa:Kullback-Leibler-divergence} \emph{Let 
\[
(X,\,\mathcal{P})\in\mathrm{ob}(\mathbf{Stat})
\]
Then for $P,\,Q\in\mathcal{P}$ the Kullback-Leibler divergence $D_{KL}[\cdot\parallel\cdot]:\mathcal{P}\times\mathcal{P}\to\mathbb{R}^{+}$
is defined by:
\begin{equation}
D_{\mathrm{KL}}[P\parallel Q]\coloneqq\int_{X}\mathrm{d}_{\mu}P(x)\log\frac{\mathrm{d}_{\mu}P(x)}{\mathrm{d}_{\mu}Q(x)}\mathrm{d}\mu(x)\label{eq:divergence:kl}
\end{equation}
}\textbf{Remark}:\emph{ The Kullback-Leibler divergence measures the
amount of information, which is gained when one revises ones beliefs
from the prior probability distribution $P$ to the posterior probability
distribution $Q$.}
\end{example*}
Riemannian statistical manifolds \emph{$(X,\,\mathcal{P},\,D)$} are
identified with Riemannian manifolds $(M,\,g)$, where $M$ is given
by $(X,\,\mathcal{P})$ and the Riemannian metric $g$ by the linearisation
of $D$. This identification is well-defined, since $(X,\,\mathcal{P})$
is by definition a smooth manifold and the linearisation of a locally
linear divergence for any $P\in\mathcal{P}$ yields a positive definite
matrix $G(P)$. Let $\xi$ be a differentiable parametrisation of
$(X,\,\mathcal{P})$, then the line element $\mathrm{d}s^{2}$ has
a local representation:
\begin{equation}
\mathrm{d}s_{P}^{2}=2D[P\parallel P+\mathrm{d}P]=\mathrm{d}\boldsymbol{\xi}^{T}G_{\xi}(P_{\xi})\mathrm{d}\boldsymbol{\xi}\label{eq:manifold:metric:idistance}
\end{equation}
This allows the determination of distances in \emph{$(X,\,\mathcal{P})$}
by the length of continuously differentiable curves.
\begin{lem}
\label{lem:3.6}Let $(X,\,\mathcal{P},\,D)$ be a Riemannian statistical
manifold and $P,\,Q\in\mathcal{P}$. Then the length of continuously
differentiable curves $\gamma_{P,Q}:[a,\,b]\to(X,\,\mathcal{P})$
with $\gamma(a)=P$ and $\gamma(b)=Q$ has a unique infimum $d_{P,Q}$,
such that $d_{P,Q}\geq0$ and $d_{P,Q}=0\Leftrightarrow P=Q$. 
\end{lem}

\begin{proof}
Since $D$ is locally linear, it follows that for any curve $\gamma$
from $P$ to $Q$ it holds, that: 
\begin{equation}
L(\gamma_{P,Q})\geq D[\gamma(a)\parallel\gamma(b)]\geq0\label{eq:manifold:metric:geodesic_vs_divergence}
\end{equation}
Therefore the length of all continuously differentiable curves from
$P$ to $Q$ has a unique infimum $d_{P,Q}\geq0$. Let $P=Q$, then
the continuously differentiable curves connecting $P$ and $Q$ may
be contracted at $P$ such that $d_{P,Q}=\inf L(\gamma_{P,Q})=\lim\inf_{Q\to P}D[P\parallel Q]=D[P\parallel P]=0$.
Conversely let $P\ne Q$, then $L(\gamma_{P,Q})\geq D[P\parallel Q]>0$. 
\end{proof}
Then for $P,\,Q\in\mathcal{P}$ a geodesic from $P$ to $Q$ is given
by a continuous differentiable curve $\gamma_{P,Q}:[a,\,b]\to(X,\,\mathcal{P})$
with $\gamma_{P,Q}(a)=P$ and $\gamma_{P,Q}(b)=Q$, such that $\gamma_{P,Q}$
minimizes the length among all continuously differentiable curves
from $P$ to $Q$. In the purpose to preserve the distance of the
Riemannian structure, the parametrisation of $(X,\,\mathcal{P},\,D)$
requires to preserve the property of a curve to be a geodesic within
the parameter space. These preserved geodesics are termed affine geodesics.
If a differentiable parametrisation globally preserves the distances,
then it is given by an isometric embedding of \emph{$(X,\,\mathcal{P},\,D)$}
and termed an affine parametrisation.
\begin{defn*}[Affine parametrisation]
\label{def:Affine-parametrisation}\emph{ Let $(X,\,\mathcal{P},\,D)$
be a Riemannian statistical manifold. Then a parametrisation $\xi$
is termed an affine parametrisation for $(X,\,\mathcal{P},\,D)$,
iff the geodesics in $(X,\,\mathcal{P},\,D)$ are $\xi$-affine geodesics.
}\textbf{Remark}:\emph{ A Riemannian statistical manifolds, given
by the notation $(X,\,\mathcal{P}_{\xi},\,D)$ implicates an affine
parametrisation $\xi$.}
\end{defn*}
The embedding of a smooth statistical manifold $(X,\,\mathcal{Q})$
within a Riemannian statistical manifold $(X,\,\mathcal{P},\,D)$
naturally induces the Riemannian metric to the submanifold $(X,\,\mathcal{Q})$,
such that also $(X,\,\mathcal{Q},\,D)$ is a Riemannian statistical
manifold. This is of particularly importance for the approximation
of high dimensional statistical models by lower dimensional submanifolds.
For this purpose the Riemannian metric is fundamental to obtain a
projection from probability distributions in $\mathcal{P}$ to their
closest approximation in $\mathcal{Q}$. This projection is a geodesic
projection:
\begin{defn*}[Geodesic projection]
\label{def:Geodesic-projection} \emph{Let $(X,\,\mathcal{P},\,D)$
be a Riemannian statistical manifold and $(X,\,\mathcal{Q})$ a smooth
submanifold. Then a mapping $\pi:\mathcal{P}\longrightarrow\mathcal{Q}$
is termed a geodesic projection, iff any point $P\in\mathcal{P}$
is mapped to a point $\pi(P)\in\mathcal{Q}$, that minimizes the distance
$d(P,\,\pi(P))$. }\textbf{Remark:}\emph{ By it's definition $d(P,\,\pi(P))<\infty$
iff $P$ and $\pi(P)$ are path-connected. Therefore geodesic projections
are by convention restricted to the common topological components
of $\mathcal{P}$ and $\mathcal{Q}$.}
\end{defn*}

\section{Dually flat Structure}

In Riemannian manifolds the problem to determine geodesic projections
to submanifolds is generally hard to solve, since the distance has
to be minimized over all continuously differentiable curves that connect
points to the submanifold. A particular convenient geometry however
arises by a flat Riemannian metric, whereas the flatness of the metric
is related to the direction of curves. The claim for a further flat
structure, which is given by the dual Riemannian metric allows a generalization
of the Pythagorean theorem and therefore an explicit calculation rule
for geodesic projections by dual affine linear projections. 
\begin{defn*}[Dual Riemannian metric]
\label{def:Dual-Riemannian-metric} \emph{Let $(M,\,g)$ be a Riemannian
manifold, then the Riemannian metric tensor $g$ is given by a family
of positive definite matrices $\{g_{P}\}_{P\in M}$. Then metric $g^{*}$,
with is dual to $g$ is given by the family of the inverse Riemannian
metric tensors, such that $g_{P}^{*}=g_{P}^{-1},\,\forall P\in M$.}
\end{defn*}
The dual Riemannian metric $g^{*}$ may be regarded as the Riemannian
metric with a locally inverse direction. In Riemannian statistical
manifolds this definition corresponds to the dual divergence.
\begin{defn*}[Dual divergence]
\label{def:Dual-divergence} \emph{Let 
\[
(X,\,\mathcal{P})\in\mathrm{ob}(\mathbf{Stat})
\]
and $D$ be a locally linear divergence over $(X,\,\mathcal{P})$.
Then the dual divergence $D^{*}$ w.r.t. $D$ is given by:
\begin{equation}
D^{*}[P\parallel Q]=D[Q\parallel P],\,\forall P,\,Q\in\mathcal{P}
\end{equation}
}
\end{defn*}
In its most simple case the Riemannian metric $g$, induced by the
divergence $D$ equals the dual Riemannian metric $g^{*}$, induced
by the dual divergence $D^{*}$. In this case the Riemannian metric
and the divergence are termed self-dual. 
\begin{defn*}[Self-dual Riemannian metric]
\label{def:Self-dual-Riemannian-metric} \emph{Let $(M,\,g)$ be
a Riemannian manifold, then the Riemannian metric tensor $g$ is termed
self-dual, iff $g^{*}=g$.}
\end{defn*}

\subsection{Dual parametrisation and Legendre transformation}

By considering a differentiable statistical manifold \emph{$(X,\,\mathcal{P}_{\xi})$}
and a real valued differentiable convex function $\psi:\mathrm{img}\xi\rightarrow\mathbb{R}$,
the differentiability of $\psi$ may be used to introduce a further
differentiable parametrisation of $(X,\,\mathcal{P})$ by $\boldsymbol{\xi}_{P}^{*}\coloneqq\nabla_{\xi}\psi(\boldsymbol{\xi}_{P})$.
Furthermore since $\psi$ is convex, the Jacobian determinant is positive
for any $\boldsymbol{\xi}_{P}\in\mathrm{dom}\xi$ and therefore the
transformation $\xi\to\xi^{*}$ is globally invertible. This defines
a bijective relationship between parameter vectors $\boldsymbol{\xi}_{P}\in\mathrm{dom}\xi$
and their respectively normal vectors in the tangent space, given
by $\boldsymbol{\xi}_{P}^{*}\in\mathrm{dom}\xi^{*}$. Since $\xi$
is an identifiable parametrisation and the transformation $\xi\to\xi^{*}$
is globally invertible, also $\xi^{*}$ is an identifiable parametrisation.
This justifies the following definition:
\begin{defn*}[Dual parametrisation]
\label{def:Dual-parametrisation} \emph{Let $(X,\,\mathcal{P}_{\xi})$
be a differentiable statistical manifold and $\psi:\mathrm{dom}\xi\rightarrow\mathbb{R}$
a sufficiently differentiable convex function. Then the dual parametrisation
for $(X,\,\mathcal{P})$ w.r.t. $\psi$ is given by:
\begin{align}
\boldsymbol{\xi}_{P}^{*} & =\nabla_{\xi}\psi(\boldsymbol{\xi}_{P}),\,\forall P\in\mathcal{P}\label{eq:def:dualparametrisation:1}
\end{align}
}
\end{defn*}
Due to the convexity of $\psi$, also the inverse transformation $\xi^{*}\to\xi$
may be represented by the partial derivation of dual function $\psi^{*}:\mathrm{dom}\xi^{*}\rightarrow\mathbb{R}$.
This yields a transformation $(\xi,\,\psi)\to(\xi^{*},\,\psi^{*})$,
which is defined by a dualistic relationship between $(\xi,\,\psi)$
and $(\xi^{*},\,\psi^{*})$, such that additional to equation \ref{eq:def:dualparametrisation:1}
also:
\begin{align}
\boldsymbol{\xi}_{P} & =\nabla_{\xi^{*}}\psi^{*}(\boldsymbol{\xi}_{P}^{*}),\,\forall P\in\mathcal{P}\label{eq:def:legendretransformation:1}
\end{align}
This transformation is known as the \emph{Legendre transformation}
and the function $\psi^{*}$ as the \emph{Legendre dual function}
of $\psi$.
\begin{lem}
\label{lem:3.7}Let $(X,\,\mathcal{P}_{\xi})$ be a differentiable
statistical manifold, $\psi:\mathrm{dom}\xi\rightarrow\mathbb{R}$
a differentiable convex function and $(\xi,\,\psi)\to(\xi^{*},\,\psi^{*})$
a Legendre transformation of $(\xi,\,\psi)$, then the Legendre dual
function $\psi^{*}:\mathrm{dom}\text{\ensuremath{\xi}}^{*}\rightarrow\mathbb{R}$
is given by:
\begin{equation}
\psi^{*}(\boldsymbol{\xi}_{P}^{*})=\arg\max_{P}\left(\boldsymbol{\xi}_{P}\cdot\boldsymbol{\xi}_{P}^{*}-\psi(\boldsymbol{\xi}_{P})\right),\,\forall P\in\mathcal{P}\label{eq:dualfunc}
\end{equation}
\end{lem}

\begin{proof}
By applying the definition of the dual parametrisation $\xi^{*}$
it has only to be proofed, that the function $\psi^{*}$, given by
$\ref{eq:dualfunc}$ satisfies the conditions of the Legendre dual
function, given by equation \ref{eq:def:legendretransformation:1}.
Let $P\in\mathcal{P}$, then:
\begin{eqnarray*}
 &  & \nabla_{\xi^{*}}\psi^{*}(\boldsymbol{\xi}_{P}^{*})\\
 &  & \stackrel{\ref{eq:dualfunc}}{=}\boldsymbol{\xi}_{P}+\left(\partial_{\xi^{*}}\boldsymbol{\xi}_{P}\right)\cdot\boldsymbol{\xi}_{P}^{*}-\nabla_{\xi}\psi(\boldsymbol{\xi}_{P})\cdot\left(\partial_{\xi^{*}}\boldsymbol{\xi}_{P}\right)\\
 &  & =\boldsymbol{\xi}_{P}+\left(\partial_{\xi^{*}}\boldsymbol{\xi}_{P}\right)\cdot\nabla_{\xi}\psi(\boldsymbol{\xi}_{P})-\nabla_{\xi}\psi(\boldsymbol{\xi}_{P})\cdot\left(\partial_{\xi^{*}}\boldsymbol{\xi}_{P}\right)\\
 &  & =\boldsymbol{\xi}_{P}+\left(\nabla_{\xi}\psi(\boldsymbol{\xi}_{P})-\nabla_{\xi}\psi(\boldsymbol{\xi}_{P})\right)\cdot\left(\partial_{\xi^{*}}\boldsymbol{\xi}_{P}\right)=\boldsymbol{\xi}_{P}
\end{eqnarray*}
\end{proof}

\subsection{Bregman divergence}

The dualistic relationship between dual parametrisations, given by
the Legendre transformation shall be extended to Riemannian metrices.
To this end a family of locally linear divergences is introduced,
that generates a dualistic relationship structure: 
\begin{defn*}[Bregman divergence]
\label{def:Bregman-divergence} \emph{Let $(X,\,\mathcal{P}_{\xi})$
be a differentiable statistical manifold and $\psi:\mathrm{dom}\xi\rightarrow\mathbb{R}$
a differentiable convex function. Then for $P,\,Q\in\mathcal{P}$
the }\textbf{\emph{Bregman divergence}}\emph{ $D_{\psi}$ w.r.t. the
differentiable parametrisation $\xi$ is given by:
\begin{equation}
D_{\psi}[P\parallel Q]=\psi(\boldsymbol{\xi}_{P})-\psi(\boldsymbol{\xi}_{Q})-\nabla_{\xi}\psi(\boldsymbol{\xi}_{P})\cdot(\boldsymbol{\xi}_{Q}-\boldsymbol{\xi}_{P})\label{eq:def:Bregman-divergence:1}
\end{equation}
}
\end{defn*}
\begin{lem}
\label{lem:3.8}Let \emph{$D_{\psi}$ }be a Bregman divergence with
regard to a sufficiently differentiable parametrisation $\xi$, then
\emph{$D_{\psi}$} is locally linear and the Riemannian metric, induced
by \emph{$D_{\psi}$,} is given by:
\begin{equation}
g_{P}=\mathrm{\nabla_{\xi}^{2}}\psi(\boldsymbol{\xi}_{P})\label{eq:lem:3.8:1}
\end{equation}
\end{lem}

\begin{proof}
By applying the definition of a differentiable convex function it
follows, that $D_{\psi}$ is locally linear and the linearisation
term of the Taylor expansion yields $G_{\xi}(P_{\xi})=\nabla_{\xi}^{2}\psi(\boldsymbol{\xi}_{P})$.
Since by the definition of a Riemannian statistical manifold $g_{P}=G_{\xi}(P_{\xi})$,
it follows that $\mathrm{g_{P}=\nabla_{\xi}^{2}}\psi(\boldsymbol{\xi}_{P})$
\end{proof}
\begin{lem}
\label{lem:3.9}Let $D_{\psi}$ be a Bregman divergence, then the
dual divergence $D_{\psi}^{*}$ is given by the Bregman divergence
of the Legendre dual function $\psi^{*}$, such that:
\begin{equation}
D_{\psi}^{*}[P\parallel Q]=D_{\psi^{*}}[P\parallel Q]\label{eq:divergence_dual_divergence}
\end{equation}
\end{lem}

\begin{proof}
Let $G_{\xi}(P_{\xi})$ be the linearisation of $D_{\psi}$ at $P\in\mathcal{P}$,
then:
\begin{eqnarray*}
 &  & G_{\xi}(P_{\xi})\\
 &  & =\nabla_{\xi}^{2}\psi(\boldsymbol{\xi}_{P})=\nabla_{\xi}\boldsymbol{\xi}_{P}^{*}\\
 &  & =(\nabla_{\xi^{*}})^{-1}\boldsymbol{\xi}_{P}=(\nabla_{\xi^{*}}^{2}\psi^{*})^{-1}(\boldsymbol{\xi}_{P}^{*})\\
 &  & =G_{\xi^{*}}^{-1}(P_{\xi^{*}})
\end{eqnarray*}
And therefore:
\begin{equation}
G_{\xi}(P_{\xi})=G_{\xi^{*}}^{-1}(P_{\xi^{*}}),\,\forall P\in\mathcal{P}\label{eq:divergence_dual_hessian}
\end{equation}
Since $D_{\psi}$ is locally linear $G_{\xi}(P_{\xi})$ is positive
definite $\forall P\in\mathcal{P}$. From equation \ref{eq:divergence_dual_hessian}
it therefore follows, that also $G_{\xi^{*}}^{-1}(P_{\xi^{*}})$ is
positive definite $\forall P\in\mathcal{P}$ and since the inverse
matrix of a positive definite matrix is also positive definite it
follows, that $G_{\xi^{*}}(P_{\xi^{*}})$ is a positive definite $\forall P\in\mathcal{P}$.
Furthermore by the definition of the Legendre transformation $G_{\xi^{*}}(P_{\xi^{*}})$
is the Hessian matrix of $\psi^{*}(\boldsymbol{\xi}_{P}^{*})$ and
therefore $\psi^{*}$ is a convex function of $\boldsymbol{\xi}_{P}^{*}\in\mathrm{dom}\xi^{*}$.
Therefore $\psi^{*}$satisfies the requirement for the definition
of a Bregman divergence. Let $P,\,Q\in\mathcal{P}$, then:
\begin{eqnarray*}
 &  & D_{\psi^{*}}[P\parallel Q]\\
 &  & \stackrel{}{=}\psi^{*}(\boldsymbol{\xi}_{P}^{*})-\psi^{*}(\boldsymbol{\xi}_{Q}^{*})-\nabla_{\xi^{*}}\psi^{*}(\boldsymbol{\xi}_{Q}^{*})(\boldsymbol{\xi}_{Q}^{*}-\boldsymbol{\xi}_{P}^{*})\\
 &  & \stackrel{}{=}\psi(\boldsymbol{\xi}_{Q})-\psi(\boldsymbol{\xi}_{P})-\nabla_{\xi}\psi(\boldsymbol{\xi}_{P})(\boldsymbol{\xi}_{P}-\boldsymbol{\xi}_{Q})\\
 &  & \stackrel{\mathrm{def}}{=}D_{\psi}[Q\parallel P]\\
 &  & \stackrel{\mathrm{def}}{=}D_{\psi}^{*}[P\parallel Q]
\end{eqnarray*}
\end{proof}
\begin{prop}
\label{prop:3.3}Let $(X,\,\mathcal{P}_{\xi},\,D_{\psi})$ be a Riemannian
statistical manifold with a Bregman divergence $D_{\psi}$. Then the
dual Riemannian metric $g^{*}$ is induced by the Bregman divergence
$D_{\psi^{*}}$ of the Legendre dual function $\psi^{*}$.
\end{prop}

\begin{proof}
By applying the definition for the dual Riemannian metric for $P\in\mathcal{P}$
it follows, that::
\[
g_{P}^{*}\stackrel{\mathrm{def}}{=}g_{P}^{-1}\stackrel{\mathrm{def}}{=}G_{\xi}^{-1}(P_{\xi})\stackrel{\ref{eq:divergence_dual_hessian}}{=}G_{\xi}^{*}(P_{\xi})\stackrel{\ref{eq:lem:3.8:1}}{=}\nabla_{\xi^{*}}^{2}\psi^{*}(\boldsymbol{\xi}_{P}^{*})
\]
This is the linearisation of the Bregman divergence $D_{\psi^{*}}$.
\end{proof}

\subsection{Dually flat statistical manifolds}
\begin{defn*}[Dually flat manifold]
\label{def:Dually-flat-manifold} \emph{Let $(M,\,g)$ be a Riemannian
manifold. Then $(M,\,g)$ is termed a dually flat (Riemannian) manifold,
iff:}
\begin{enumerate}
\item[(1)] \emph{ $g$ is a flat Riemannian metric of $M$}
\item[(2)] \emph{$g^{*}$ is a flat Riemannian metric of $M$}
\end{enumerate}
\end{defn*}
\begin{example*}[Euclidean space]
\label{exa:Euclidean-space} An example for a dually flat manifold
is given by Euclidean spaces. Let $\xi$ be the Cartesian coordinates
of an Euclidean space $E$, then $\xi$ is an affine parametrisation
of $E$ and for $P,\,Q\in E$ the Euclidean metric of $E$ is induced
by the Euclidean divergence $D[P\parallel Q]=\frac{1}{2}|\boldsymbol{\xi}_{P}-\boldsymbol{\xi}_{Q}|^{2}$.
In this case the divergence is symmetric and therefore the Euclidean
metric is self-dual Since $E$ is flat with regard to the Euclidean
metric $E$ is also flat with regard to the dual Euclidean metric
and therefore $E$ is a dually flat manifold. 
\end{example*}
\begin{lem}
\label{lem:3.10}Let $(X,\,\mathcal{P}_{\xi},\,D_{\psi})$ be a Riemannian
statistical manifold with a Bregman divergence $D_{\psi}$. Then $(X,\,\mathcal{P_{\xi}},\,D_{\psi})$
is a dually flat statistical manifold, iff:
\begin{enumerate}
\item[(1)] \emph{}The $\xi$-affine geodesics are flat with regard to the Riemannian
metric, induced by $D_{\psi}$
\item[(2)] \emph{}The $\xi^{*}$-affine geodesics are flat with regard to the
Riemannian metric, induced by $D_{\psi^{*}}$
\end{enumerate}
\end{lem}

\begin{proof}
Since the Legendre transformation generally does not preserve the
Riemannian metric, the flatness of $(X,\,\mathcal{P},\,D_{\psi})$
and $(X,\,\mathcal{P},\,D_{\psi^{*}})$ are indeed independent properties.
Let the $\xi$-affine geodesics be flat with regard to the Riemannian
metric $g$, induced by $D_{\psi}$, then also $(X,\,\mathcal{P})$
is flat w.r.t. $g$. Let further the $\xi^{*}$-affine geodesics be
flat with regard to the Riemannian metric $\tilde{g}$, induced by
$D_{\psi^{*}}$, then by \prettyref{prop:3.3} it follows, that $\tilde{g}=g^{*}$
and therefore $g^{*}$ is a flat Riemannian metric of $(X,\,\mathcal{P})$.
Conversely let $(X,\,\mathcal{P_{\xi}},\,D_{\psi})$ be a dually flat
statistical manifold with a Bregman divergence $D_{\psi}$. Then by
convention $\xi$ is an affine parametrisation of $(X,\,\mathcal{P},\,D_{\psi})$
and the geodesics in $(X,\,\mathcal{P},\,D_{\psi})$ are $\xi$-affine
geodesics and flat with regard to the Riemannian metric, induced by
$D_{\psi}$. Furthermore the dual Riemannian metric $g^{*}$ induced
by $D_{\psi}^{*}$ is a flat Riemannian metric of $(X,\,\mathcal{P})$
and since $D_{\psi}$ is a Bregman divergence it follows that $D_{\psi}^{*}=D_{\psi^{*}}$.
Then the dual parametrisation $\xi^{*}$is an affine parametrisation
of $(X,\,\mathcal{P},\,D_{\psi}^{*})$ and the geodesics in $(X,\,\mathcal{P},\,D_{\psi})$
are $\xi^{*}$-affine geodesics and flat with regard to the Riemannian
metric, induced by $D_{\psi^{*}}$ .
\end{proof}
\begin{defn*}[Dual geodesic projection]
\label{def:Dual-geodesic-projection} \emph{Let $(X,\,\mathcal{P},\,D)$
be a Riemannian statistical manifold with a smooth submanifold $(X,\,\mathcal{Q})$.
Then a mapping $\pi^{*}:\mathcal{P}\longrightarrow\mathcal{Q}$ is
termed a }\textbf{\emph{dual geodesic projection}}\emph{, iff any
point $P\in\mathcal{P}$ is mapped to a point $\pi^{*}(P)\in\mathcal{Q}$,
that minimizes the distance $d(P,\,\pi^{*}(P))$ w.r.t. the dual Riemannian
metric, which is induced by $D^{*}$.}
\end{defn*}
In the case of a dually flat statistical manifold, the dual affine
structure induces a correspondence relationship between the Riemannian
metrices, induced by $D$ and $D^{*}$.
\begin{lem}
\label{lem:3.11}Let $(X,\,\mathcal{P}_{\xi},\,D_{\psi})$ be a Riemannian
statistical manifold with a Bregman divergence $D_{\psi}$. Then $D_{\psi}$
has a mixed representation in the parametrisations $\xi$ and $\xi^{*}$,
which is given by: 
\begin{equation}
D_{\psi}[P\parallel Q]=\psi(\boldsymbol{\xi}_{P})+\psi^{*}(\boldsymbol{\xi}_{Q}^{*})-\boldsymbol{\xi}_{P}\cdot\boldsymbol{\xi}_{Q}^{*}\label{eq:divergence_bregman_mixed}
\end{equation}
\end{lem}

\begin{proof}
By applying the definition of the dual divergence and \prettyref{lem:3.9}
it follows, that:
\[
D_{\psi}[P\parallel Q]\stackrel{\ref{eq:divergence_dual_divergence}}{=}D_{\psi^{*}}[Q\parallel P]
\]
The right side of the equation is calculated by the definition of
the Bregman divergence and the Legendre dual function, such that:
\begin{eqnarray*}
 &  & D_{\psi^{*}}[Q\parallel P]\\
 &  & \stackrel{\ref{eq:def:Bregman-divergence:1}}{=}\psi^{*}(\boldsymbol{\xi}_{Q}^{*})-\psi^{*}(\boldsymbol{\xi}_{P}^{*})-\nabla\psi^{*}(\boldsymbol{\xi}_{P}^{*})(\boldsymbol{\xi}_{P}^{*}-\boldsymbol{\xi}_{Q}^{*})\\
 &  & \stackrel{}{=}\psi(\boldsymbol{\xi}_{P})+\psi^{*}(\boldsymbol{\xi}_{Q}^{*})-\boldsymbol{\xi}_{P}\cdot\boldsymbol{\xi}_{Q}^{*}
\end{eqnarray*}
 
\end{proof}
\begin{thm}[\emph{Amari Pythagorean Theorem}]
\emph{\label{thm:Pythagorean-theorem}} Let $(X,\,\mathcal{P_{\xi}},\,D_{\psi})$
be a dually flat statistical manifold, which is given by a Bregman
divergence $D_{\psi}$ and let $P,\,Q,\,R\in\mathcal{P}$ be an orthogonal
triangle in the sense, that the $\xi^{*}$-affine geodesic $\gamma_{P,Q}^{*}$
from $P$ to $Q$ is orthogonal to the $\xi$-affine geodesic $\gamma_{Q,R}$
from $Q$ to $R$, then:
\begin{equation}
D_{\psi}[P\parallel R]=D_{\psi}[P\parallel Q]+D_{\psi}[Q\parallel R]\label{eq:thm:Pythagorean_theorem:1}
\end{equation}
\begin{figure}[h]
\begin{centering}
\def\svgwidth{\columnwidth} 
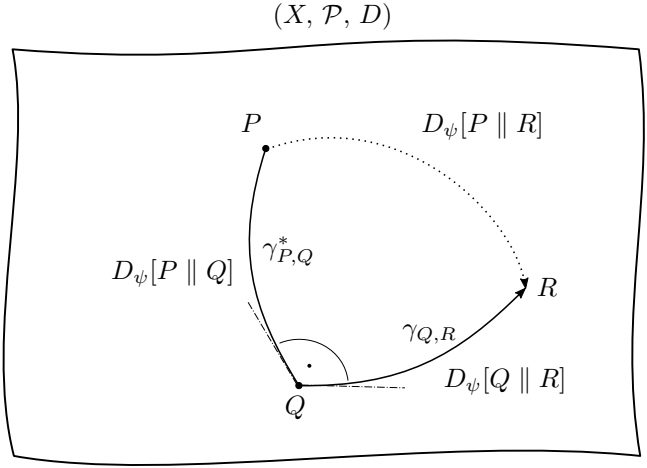
\par\end{centering}
\caption{\label{fig:thm:Pythagorean-theorem}Pythagorean theorem for dually
flat manifolds}
\end{figure}
\end{thm}

\begin{proof}
The $\xi^{*}$-affine geodesic $\gamma_{P,Q}^{*}:[0,\,1]\to(X,\,\mathcal{P})$,
with $\gamma_{P,Q}^{*}(0)=P$ and $\gamma_{P,Q}^{*}(1)=Q$ is parametrized
by:
\[
\boldsymbol{\xi}_{P,Q}^{*}(t)=t\boldsymbol{\xi}_{Q}^{*}+(1-t)\boldsymbol{\xi}_{P}^{*},\,t\in[0,\,1]
\]
and the $\xi$-affine geodesic $\gamma_{Q,R}:[0,\,1]\to(X,\,\mathcal{P})$,
with $\gamma_{Q,R}(0)=Q$ and $\gamma_{Q,R}(1)=R$ by:
\[
\boldsymbol{\xi}_{Q,R}(t)=t\boldsymbol{\xi}_{R}+(1-t)\boldsymbol{\xi}_{Q},\,t\in[0,\,1]
\]
Let $\langle\cdot,\,\cdot\rangle_{g}$ denote the local scalar product,
which is induced by the Bregman divergence $D_{\psi}$. By applying
the definition of the Bregman divergence, the local scalar product
at the point $Q$ is given by:
\begin{eqnarray*}
 &  & \langle\frac{\mathrm{d}}{\mathrm{d}t}\gamma_{P,Q}^{*}(t)|_{t=1},\,\frac{\mathrm{d}}{\mathrm{d}t}\gamma_{Q,R}(t)|_{t=0}\rangle_{g}\\
 &  & \stackrel{\ref{eq:def:Bregman-divergence:1}}{=}(\boldsymbol{\xi}_{Q}^{*}-\boldsymbol{\xi}_{P}^{*})\cdot(\boldsymbol{\xi}_{R}-\boldsymbol{\xi}_{Q})\\
 &  & \stackrel{\ref{eq:divergence_bregman_mixed}}{=}\boldsymbol{\xi}_{Q}^{*}\cdot\boldsymbol{\xi}_{R}-\boldsymbol{\xi}_{P}^{*}\cdot\boldsymbol{\xi}_{R}+\boldsymbol{\xi}_{P}^{*}\cdot\boldsymbol{\xi}_{Q}-\psi(\boldsymbol{\xi}_{Q})-\psi^{*}(\boldsymbol{\xi}_{Q}^{*})\\
 &  & \stackrel{\ref{eq:divergence_bregman_mixed}}{=}D_{\psi}[P\parallel Q]+D_{\psi}[Q\parallel R]-D_{\psi}[P\parallel R]
\end{eqnarray*}
Since $\gamma_{P,Q}^{*}$ and $\gamma_{Q,R}$ are required to be orthogonal
in the point $Q$, the left side of the equation equals zero and therefore
it follows, that:
\[
D_{\psi}[P\parallel Q]+D_{\psi}[Q\parallel R]-D_{\psi}[P\parallel R]=0
\]
\end{proof}
Due to the generic asymmetry of Bregman divergences the generalized
Pythagorean theorem has a corresponding dual theorem, which mutatis
mutandis is given by:
\begin{equation}
D_{\psi}^{*}[P\parallel R]=D_{\psi}^{*}[P\parallel Q]+D_{\psi}^{*}[Q\parallel R]\label{eq:dual_flat_pythagoras_dual}
\end{equation}
If $\psi$ is chosen, such that $D_{\psi}$ is symmetric, the induced
Riemannian metric of $D_{\psi^{*}}$ is identical to that of $D_{\psi}$,
since:
\[
D_{\psi}[P\parallel Q]=D_{\psi}[Q\parallel P]=D_{\psi^{*}}[P\parallel Q]
\]
In this case the generalized Pythagorean theorem and its dual corresponding
are equivalent and the induced Riemannian metric is self-dual.
\begin{defn*}[Affine projection]
\label{def:affine_projection} \emph{Let $(X,\,\mathcal{P_{\xi}},\,D)$
be a Riemannian statistical manifold and $(X,\,\mathcal{Q})$ a smooth
submanifold. Then a projection $\pi_{\xi}^{\perp}:\mathcal{P}\to\mathcal{Q}$
is termed an }\textbf{\emph{$\xi$-affine projection}}\emph{ from
}$(X,\,\mathcal{P},\,D)$\emph{ to $(X,\,\mathcal{Q},\,D)$, iff for
any $P\in\mathcal{P}$ the $\xi$-affine geodesics from $P$ to $\pi_{\xi}^{\perp}(P)$
are orthogonal to $\mathcal{Q}.$}
\end{defn*}
\begin{lem}
\label{lem:3.12}Let $(X,\,\mathcal{P}_{\xi},\,D_{\psi})$ be a dually
flat statistical manifold with a smooth submanifold $(X,\,\mathcal{Q})$.
Then there exists an \emph{$\xi$-}affine projection as well as an
\emph{$\xi^{*}$-}affine projection from $(X,\,\mathcal{P},\,D)$\emph{
to $(X,\,\mathcal{Q})$}.
\end{lem}

\begin{proof}
Since $(X,\,\mathcal{P}_{\xi},\,D_{\psi})$ is Riemannian statistical
manifold by convention $\xi$ is an affine parametrisation and therefore
by definition any $P,\,Q\in\mathcal{P}$ are connected by a \emph{$\xi$-}affine
geodesic $\gamma_{P,Q}:[0,\,1]\to(X,\,\mathcal{P})$ with $\gamma_{P,Q}(0)=P$
and $\gamma_{P,Q}(1)=Q$. Let's assume, that for a given $P\in\mathcal{P}$
there is no $Q\in\mathcal{Q}$, such that $\gamma_{P,Q}\bot\mathcal{Q}$,
then due to the \emph{mean value theorem} $\mathcal{Q}$ is not differentiable
with regard to the affine parametrisation $\xi$ and since $\xi$
is a homeomorphism $\mathcal{Q}$ is also not differentiable in $\mathcal{P}$.
However since $\mathcal{Q}$ is a smooth submanifold this does not
hold, such that there exists a $Q\in\mathcal{Q}$ with $\gamma_{P,Q}\bot\mathcal{Q}$.
The argument is true for any $P\in\mathcal{P}$ and therefore proves
the existence of an \emph{$\xi$-}affine projection. Since $(X,\,\mathcal{P}_{\xi},\,D_{\psi})$
is dually flat also $\xi^{*}$ is an affine parametrisation. Then
the argument, given for the \emph{$\xi$-}affine projection mutatis
mutandis proves the existence of an \emph{$\xi^{*}$-}affine projection
is argument may analogous be applied to the dual space, and the also
proves the existence of a dual affine projection.
\end{proof}
\begin{thm}[\emph{Amari Projection theorem}]
\label{thm:Projection-theorem}\emph{} Let $(X,\,\mathcal{P}_{\xi},\,D_{\psi})$
be a dually flat statistical manifold and $(X,\,\mathcal{Q})$ a smooth
submanifold. Then the geodesic projection $\pi:\mathcal{P}\to\mathcal{Q}$
is an \emph{$\xi^{*}$}-affine projection and the dual geodesic projection
$\pi^{*}:\mathcal{P}\to\mathcal{Q}$ is an $\xi$-affine projection.
\begin{figure}[h]
\begin{centering}
\def\svgwidth{\columnwidth} 
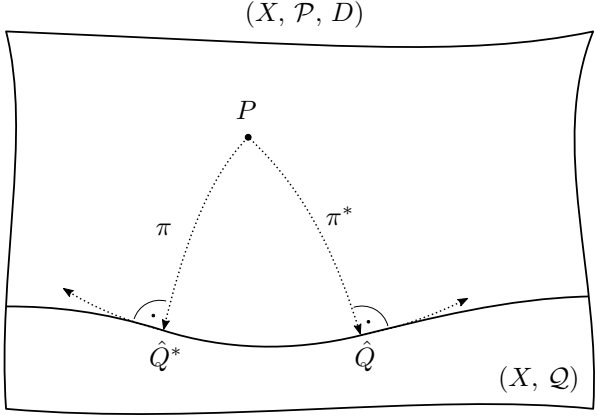
\par\end{centering}
\caption{\label{fig:thm:Projection-theorem}Projection theorem for dually flat
manifolds}
\end{figure}
\end{thm}

\begin{proof}
Let $P\in\mathcal{P}$, then due to \prettyref{lem:3.12} a dual affine
projection $\pi_{\xi^{*}}^{\perp}$ may be chosen, such that the dual
affine curve from $P_{\xi^{*}}$ to $\pi_{\xi^{*}}^{\perp}(P)$ is
orthogonal to $\mathcal{Q}_{\xi^{*}}$. Let $Q=\pi_{\xi^{*}}^{\perp}(P)$.
Then for any sufficiently close $R=Q+\mathrm{d}Q\in\mathcal{Q}$ with
$\boldsymbol{\xi}_{R}=\boldsymbol{\xi}_{\hat{Q}}+\mathrm{d}\boldsymbol{\xi}$
and $\mathrm{d}\boldsymbol{\xi}\ne0$ the triangle $P,\,Q,\,R\in\mathcal{P}$
is orthogonal in $\mathcal{Q}$ and \prettyref{thm:Pythagorean-theorem}
gives the relation $D[P\parallel R]>D[P\parallel Q]$. This shows,
that $Q$ is a critical point w.r.t. the divergence $D[P\parallel Q]$.
Conversely since $\mathcal{Q}$ is a smooth submanifold the mean value
theorem shows that for any critical point $Q\in\mathcal{Q}$, w.r.t.
the divergence $D[P\parallel Q]$ a dual affine projection from $P$
to $Q$ exists and therefore in particular for the points $\hat{Q}\in\mathcal{Q}$
that minimizes the divergence. From equation \ref{eq:manifold:metric:geodesic_vs_divergence}
we obtain for the distance that $D[P\parallel\hat{Q}]\leq d(P,\,\hat{Q})$.
Furthermore by definition $d(P,\,\hat{Q})$ is the minimal length
of a curve from $P$ to $\hat{Q}$, but since there exists a dual
affine projection from $P$ to $\hat{Q}$, which has the length $D[P\parallel\hat{Q}]$
it follows that $d(P,\,\hat{Q})=D[P\parallel\hat{Q}]$ and therefore
the geodesic projection is a dual affine projection. By applying the
dual version of \prettyref{thm:Pythagorean-theorem} this argument
mutatis mutandis also holds for the dual geodesic projection w.r.t.
the affine projection.
\end{proof}
\begin{cor}
\label{cor:3.3}Let $(X,\,\mathcal{P}_{\xi},\,D_{\psi})$ be a dually
flat statistical manifold and $(X,\,\mathcal{Q})$ and $(X,\,\mathcal{S})$
smooth submanifolds. Let further be $(X,\,\mathcal{Q})$ flat w.r.t.
$D_{\psi^{*}}$ and $(X,\,\mathcal{S})$ flat w.r.t. $D_{\psi}$.
Then the geodesic projection $\pi:\mathcal{P}\to\mathcal{Q}$ is uniquely
given by an \emph{$\xi^{*}$}-affine projection and the dual geodesic
projection $\pi^{*}:\mathcal{P}\to\mathcal{S}$ is uniquely given
by an \emph{$\xi$}-affine projection. 
\begin{figure}[h]
\begin{centering}
\def\svgwidth{\columnwidth} 
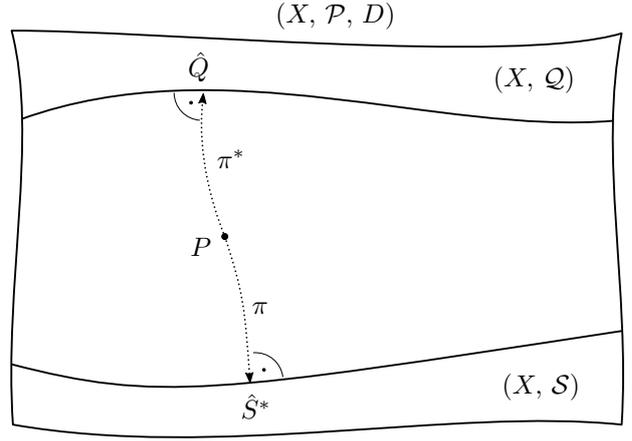
\par\end{centering}
\caption{\label{fig:cor:3.2}Unique projections in dually flat manifolds}
\end{figure}
\end{cor}

\begin{proof}
By virtue of \prettyref{thm:Projection-theorem} it suffices to proof
the uniqueness of the affine projection and the dual affine projection.
Let $p\in\mathcal{P}$, then \prettyref{lem:3.12} asserts the existence
of a dual affine projection of $P$ to a point $\pi(P)=\hat{Q}\in\mathcal{Q}$
and since $(X,\,\mathcal{Q})$ is flat it follows, that $\mathcal{Q}\subseteq T_{Q}\mathcal{Q}$
such that for any $R\in\mathcal{Q}$ \prettyref{thm:Pythagorean-theorem}
shows that:
\[
D[P\parallel R]=D[P\parallel Q]+D[Q\parallel R]\geq D[P\parallel Q]
\]
Therefore $Q$ is the global minimum and $\pi(P)$ is unique. By the
application of the dual version of \prettyref{thm:Pythagorean-theorem}
to the submanifold $(X,\,\mathcal{S})$ the argument mutatis mutandis
also proves, that $\pi^{*}(S)=\hat{S}\in\mathcal{S}$ is unique in
$(X,\,\mathcal{S})$.
\end{proof}

\bibliographystyle{unsrt}
\bibliography{articles}

\end{document}